\documentclass[12pt]{amsart}
\usepackage[a4paper,bottom=2cm,top=1cm,left=2cm,right=1.3cm]{geometry}
\usepackage[english]{babel}
\usepackage{amsfonts,amsmath,latexsym,amssymb,amsthm,mathrsfs,upref,mathptmx} 

\newtheorem{theorem}{Theorem}           
\newtheorem{lemma}{Lemma}

\theoremstyle{definition}

\newtheorem{definition}{Definition}
\newtheorem{example}{Example}

\newtheorem{remark}{Remark}

\begin{document}

\title{Formally self-adjoint quasi-differential operators and boundary value problems}

\author{Andrii Goriunov, Vladimir Mikhailets, Konstantin Pankrashkin}

\address{Institute of Mathematics of National Academy of Sciences of Ukraine, Kyiv, Ukraine}
\email{goriunov@imath.kiev.ua}

\address{Institute of Mathematics of National Academy of Sciences of Ukraine, Kyiv, Ukraine}
\email{mikhailets@imath.kiev.ua}

\address{Laboratory of mathematics, University Paris-Sud 11, Orsay, France}
\email{konstantin.pankrashkin@math.u-psud.fr}



\thanks{The first and second authors were partially supported by the grant no. 01-01-12 
of National Academy of Sciences of Ukraine (under the joint Ukrainian-Russian project of NAS of 
Ukraine and Russian Foundation of Basic Research)}

\begin{abstract}
We develop the machinery of boundary triplets for one-dimensional operators
generated by formally self-adjoint quasi-differential expression of arbitrary order
on a finite interval. The technique are then used to describe
all maximal dissipative, accumulative 
and self-adjoint extensions of the associated minimal operator
and its generalized resolvents in terms
of the boundary conditions.
Some specific classes are considered in greater detail. 
\end{abstract}

\maketitle

\section{Introduction}

Many problems of the modern mathematical physics and the  quantum mechanics
lead to the study of differential operators with strongly singular coefficients such as
Radon measures or even more singular distributions,
see the monographs \cite{Albeverio, AlbeKur} and the very recent papers \cite{Teschl-12-unpubl,Teschl-arXiv1105.3755E,EGNT,EGNT2}
and the references therein. In such situations one is faced with the problem of a correct definition of such operators as
the classical methods of the theory of differential operators cannot be applied anymore.
It was observed in the recent years that a large class of one-dimensional operators
can be handled
in a rather efficient way with the help of the so-called Shin--Zettl quasi--derivatives \cite{AtkEvZ-88, SavShkal-99}.
The class of such operators includes, for example,
the Sturm-Liouville operators  acting
in $L_2([a,b], \mathbb{C})$ by the rule
\begin{equation}\label{St-L expr}
l(y) = - (py')' + qy,
\end{equation}
where the coefficients $p$ and $q$ satisfy the conditions  
\[
\dfrac{1}{p},\, \dfrac{Q}{p},\, \dfrac{Q^2}{p} \in L_1([a,b], \mathbb{C}),
\]
$Q$ is the antiderivative of the distribution $q$, and $[a,b]$ is a finite intevral. 
The condition $1/p \in L_1([a,b], \mathbb{C})$ implies that the potential function $q$ may be 
a finite measure on $[a,b]$, see~\cite{GMSt-L}.

For the two-term formal differential expression 
\begin{equation}\label{m expr}
l(y) = i^m y^{(m)} + qy,  \qquad m \geq 3, 
\end{equation}
where $q = Q'$ and $Q \in L_1([a,b], \mathbb{C})$,
the regularisation with quasi-derivatives 
was constructed in \cite{GM_high}. 
Similarly one can study the case
\[ q = Q^{(k)}, \qquad k \leq \left[\frac{m}{2}\right],\]
where $Q \in L_2([a,b], \mathbb{C})$ if $m$ is even and $k = m/2$, 
and $Q \in L_1([a,b], \mathbb{C})$ otherwise, 
and all the derivatives of $Q$ are understood in the sense of distributions. 

In the present paper we consider one-dimensional operators generated by the most general 
formally self-adjoint quasi-differential expression of an arbitrary order
on the Hilbert space $L_2([a,b], \mathbb{C})$, and the main result consists
in an explicit construction of a boundary triplet for the associated
symmetric minimal  quasi-differential operator.
The machinery of boundary triplets \cite{Gorb-book-eng}
is a useful tool in the description and the analysis of various boundary value problems
arising in mathematical physics, see e.g. \cite{BMN02,BGP08,DM}, and we expect that
the constructions of the present paper will be useful, in particular, in the study
of higher order differential operators on metric graphs \cite{AGA}.

The quasi-differential operators were introduced first by Shin \cite{Shin-43-eng}
and then essentially developed by Zettl \cite{Zettl-75},  see also the monograph \cite{EverMark-book} and references therein.
The paper \cite{Zettl-75} provides the description of all self-adjoint extensions 
of the minimal symmetric quasi-differential operator of even order with real-valued coefficients.
It is based on the so-called  Glasman-Krein-Naimark theory and is rather implicit.
The approach of the present work gives an explicit description of the self-adjoint extensions
as well as of all maximal dissipative/accumulative extensions in terms
of easily checkable boundary conditions.

The paper is organised as follows. In Section 1 we recall basic definitions and known facts
concerning the Shin--Zettl quasi-differential operators. 
Section 2 presents the regularization of the formal differential expressions \eqref{St-L expr} and \eqref{m expr} 
using the quasi-derivatives, and some specific examples are considered.
In Section 3 the boundary triplets for the minimal symmetric operators are constructed.
All maximal dissipative, maximal accumulative and self-adjoint extensions of these operators are 
explicitly described in terms of boundary conditions.
Section 4 deals with the formally self-adjoint quasi-differential operators with real-valued coefficients. 
We prove that every maximal dissipative/accumulative extension of the minimal operator 
in this case is self-adjoint and describe all such extensions. 
In Section 5 we give an explicit description 
of all maximal dissipative/accumulative and self-adjoint extensions 
with separated boundary conditions for a special case. 
In Section 6 we describe all generalized resolvents of the minimal operator.
Some results of this paper for some particular classes of quasi-differential expressions 
were announced without proof in \cite{GM_even, GM_odd}.

\section{Quasi-differential expressions}

In this section we recall the definition and the basic facts concerning
the Shin--Zettl quasi-derivatives and the quasi-differential operators on a finite interval, 
see \cite{EverMark-book, Zettl-75}  for a more detailed discussion.

Let $m \in \mathbb{N}$ and a finite interval $[a, b]$ be given.
Denote by $Z_m([a,b])$ the set of the $m\times m$ complex matrix-valued functions $A$
whose entires $(a_{k,s})$ satisfy
\begin{equation}\label{Quasi cond}
\begin{array}{l}
1)\,\, a_{k,s} \equiv 0, \, s > k + 1;\\
2)\,\, a_{k,s} \in L_1 \left([a,b], \mathbb{C} \right), \,
\quad a_{k, k + 1} \neq 0 \text{~a.~e.~on~} [a,b],\\
\,\,\quad k = 1,2,\ldots,m,  \, s = 1,2,\ldots,k + 1;
\end{array}
\end{equation}
such matrices will be referred to as Shin--Zettl matrices of order $m$ on $[a,b]$.
Any Shin--Zettl matrix $A$ defines recursively the associated quasi-derivatives of orders 
$k \leq m$ of a function  $y \in \operatorname{Dom}(A)$ 
in the following way:
\begin{align*}
D^{[0]}y &:= y,\\
D^{[k]}y  &:= a^{-1}_{k,k+1}(t)\left((D^{[k - 1]}y)' - \sum\limits_{s = 1}^{k} {a_{k,s}(t)D^{[s - 1]}}y\right) ,
\quad k = 1,2,\ldots,m - 1,\\
D^{[m]}y  &:= (D^{[m - 1]}y)' - \sum\limits_{s = 1}^{m} {a_{m,s}(t)D^{[s - 1]}y},
\end{align*}
and the associated domain $\operatorname{Dom}(A)$ is defined
by
\[
\operatorname{Dom}(A) := \left\{y \left| D^{[k]}y \in AC([a,b], \mathbb{C}),\, k=\overline{0,m-1}\,\right.\right\}.
\]
The above yields $D^{[m]}y \in L_1([a,b], \mathbb{C})$.
The quasi-differential expression $l(y)$ of order $m$ associated with $A$
is defined by
\begin{equation}\label{Q-d_expr}
l(y) := i^mD^{[m]}y.
\end{equation}

Let $c \in [a,b]$ and $\alpha_k \in \mathbb{C}, \, k=\overline{0, m-1}$.
We say that a function $y$ solves the Cauchy problem 
\begin{equation}\label{cauchy pr 1}
l(y) - \lambda y = f \in L_2([a,b], \mathbb{C}), \qquad
 (D^{[k]}y)(c) = \alpha_k, \quad k=\overline{0,m-1},
\end{equation}
if $y$ is the first coordinate of the vector function $w$ solving
the Cauchy problem for the associated the first order matrix equation
\begin{equation}\label{cauchy pr 2}
w'(t)=A_\lambda(t)w(t) + \varphi(t),
\qquad
w(c) = (\alpha_0, \alpha_1, \ldots \alpha_{m - 1})
\end{equation}
where we denote 
\begin{equation}\label{A matrix}
A_\lambda(t):=A(t) - \left (
\begin{array}{cccc}
0&0&\ldots &0 \\
0&0&\ldots&0 \\
\vdots  &\vdots &\ddots &\vdots  \\
0&0&\ldots &0 \\
i^{-m}\lambda&0&\ldots&0
\end{array}\right) \in L_1([a,b], \mathbb{C}^{m\times m}),
\end{equation}
and $\varphi(t) := \big(0, 0, \ldots, 0, i^{-m}f(t)\big)^T \in L_1([a,b], \mathbb{C}^m)$.
The following statement is proved in \cite{Zettl-75}. 
\begin{lemma}\label{lm_cauchy_pr_unique}
Under the assumptions \eqref{Quasi cond}, 
the problem \eqref{cauchy pr 1} has a unique solution defined on  $[a,b]$.
\end{lemma}
The quasi-differential expression $l(y)$ gives rise to the associated
\emph{maximal}
quasi-differential operator
\begin{gather*}
L_{\operatorname{max}}:y  \mapsto l(y),\\
\operatorname{Dom}(L_{\operatorname{max}}) = \left\{y \in \operatorname{Dom}(A)  \left|\,
 D^{[m]} y \in L_2([a,b], \mathbb{C})\right.\right\}.
\end{gather*}
in the Hilbert space $L_2([a,b], \mathbb{C})$, and the associated \emph{minimal} quasi-differential operator is defined as the restriction of
$L_{\operatorname{max}}$ onto the set 
\[
\operatorname{Dom}(L_{\operatorname{min}}) := 
	\left\{y \in  \operatorname{Dom}(L_{\operatorname{max}}) \left| D^{[k]}y(a) = D^{[k]}y(b) = 0,\,
k = \overline{0,m - 1}\right.\right\}.
\]
If the functions $a_{k,s}$ are sufficiently smooth, then all the brackets in the definition
of the quasi-derivatives can be expanded, and we arrive at the usual ordinary differential expressions,
and the associated quasi-differential operators become differential ones.

Let us recall the definition of the formally adjoint quasi-differential expression
$l^+(y)$.
The formally adjoint (also called the Lagrange adjoint) matrix $A^+ $ for $A \in Z_m([a,b])$ is defined by
\[
A^+ := - \Lambda_m^{-1}\overline{A^T}\Lambda_m,
\]
where $\overline{A^T}$ is the conjugate transposed matrix to $A$ and
\[
\Lambda_m :=
\begin{pmatrix}
  0 & 0  & \ldots & 0 & -1 \\
  0 & 0  & \ldots & 1 & 0 \\
  \vdots  & \vdots  & \vdots  & \vdots  & \vdots  \\
  0 & (-1)^{m - 1}  & \ldots & 0 & 0 \\
  (-1)^m & 0 & \ldots & 0 & 0
\end{pmatrix}.
\]
One can easily see that $\Lambda_m^{-1} = (-1)^{m - 1}\Lambda_m$.

We we can define the Shin--Zettl quasi-derivatives associated with $A^+$
which will be denoted by
\[
D^{\{0\}}y, D^{\{1\}}y, ..., D^{\{m\}}y,
\]
and they act on the domain 
\[
\operatorname{Dom}(A^+) := \left\{y \left| D^{\{k\}}y \in AC([a,b], \mathbb{C}),\, k=\overline{0,m-1}\,\right.\right\}.
\]
The formally adjoint quasi-differential expression is npw defined as
\[
l^+(y) := i^mD^{\{m\}}y
\]
and we denote the associated maximal and minimal operators by
$L^+_{\operatorname{max}}$ and $L^+_{\operatorname{min}}$ respectively
The following theorem is proved in \cite{Zettl-75}.
\begin{theorem}\label{thm_L_adjoint}
The operators $L_{\operatorname{min}}$, $L^+_{\operatorname{min}}$, $L_{\operatorname{max}}$, $L^+_{\operatorname{max}}$ are closed and densely defined in $L_2\left([a,b], \mathbb{C}\right)$, and satisfy
\[
L_{\operatorname{min}}^* = L^+_{\operatorname{max}},\quad L_{\operatorname{max}}^* = L^+_{\operatorname{min}}.
\]
If \, $l(y) = l^+(y)$, then the operator $L_{\operatorname{min}} = L^+_{\operatorname{min}}$ is symmetric 
with the deficiency indices $\left({m,m} \right)$, and
\[
L_{\operatorname{min}}^* = L_{\operatorname{max}},\quad L_{\operatorname{max}}^* = L_{\operatorname{min}}.
\]
\end{theorem} 
We also will require the following two lemmas whose proof can be found e.g. in \cite{EverMark-book}:
\begin{lemma}\label{lm_Lagrange}
For any $y, z \in \operatorname{Dom}(L_{\operatorname{max}})$ there holds
\[
\int\limits_a^b \left(D^{[m]}y\cdot\overline z  -
y\cdot\overline{D^{[m]}z} \right)dt =
\sum\limits_{k = 1}^{m}
(-1)^{k - 1}{D^{[m - k]}y\cdot\overline {D^{[k -1]}z}}\left|_{t = a}^{t = b}
\right.
\]
\end{lemma}

\begin{lemma}\label{lm_surj}
For any $(\alpha _0 ,\alpha _1, ..., \alpha _{m - 1}),(\beta _0,\beta _1, ..., \beta _{m - 1})\in\mathbb{C}^m$
there exists a function ${y \in \operatorname{Dom}(L_{\operatorname{max}})}$ such that
\[
D^{[k]}y(a) = \alpha _k , \quad D^{[k]}y(b) = \beta _k, \quad k = 0,1, ..., m - 1.
\]
\end{lemma}

\section{Regularizations by quasi-derivatives} 

Let us consider some classes of formal differential expressions with singular
coefficients admitting a regularisation with the help of the Shin--Zettl quasi-derivatives.

Consider first the formal Sturm--Liouville expression
\[
l(y) = -(p(t)y')'(t) + q(t)y(t), \quad t \in [a, b].
\]
The classical definition of the quasi-derivatives 
\[
D^{[0]}y := y, \qquad D^{[1]}y = py', \qquad D^{[2]}y = (D^{[1]}y)' - qD^{[0]}y
\]
allows one to interpret the above expression $l$ as a regular quasi-differential one
if the function $p$ is finite almost everywhere and, in addition,
\begin{equation}\label{St-L_cond_class}
\dfrac{1}{p}, \, q \in L_1( [a,b], \mathbb{C}).
\end{equation}
Some physically interesting coefficients $q$ (i.e. having
non-integrable singularities or being a measure) are not covered by the preceding conditions,
and this can be corrected using another set of quasi-derivatives as proposed in \cite{GM St-L, SavShkal-99}.
Set
\begin{equation}\label{St-L reg}
\begin{split}
&D^{[0]} y = y, \qquad D^{[1]} y = py' - Qy, \\
&D^{[2]} y = (D^{[1]} y)' + {\frac{Q}{p}}D^{[1]} y + {\frac{Q^2}{p}}y,
\end{split}
\end{equation}
where function $Q$ is chosen so that $Q' = q$ and the derivative is understood in the sense of distributions. 
Then the expression 
\[
l[y] = - D^{[2]} y
\]
is a Shin--Zettl quasi-differential one if the following conditions are satisfied:
\begin{equation}\label{St-L_cond_GM}
\dfrac{1}{p},\, \dfrac{Q}{p},\, \dfrac{Q^2}{p} \in L_1( [a,b], \mathbb{C}).
\end{equation}
In this case the expression $l$ generates the associated
quasi-differential operators $L_{\operatorname{min}}$ and $L_{\operatorname{max}}$.
One can easily see that if $p$ and $q$ satisfy conditions \eqref{St-L_cond_class},
then these operators  coincide with the classic Sturm-Liouville operators, but
the conditions \eqref{St-L_cond_GM} are considerably weaker than 
\eqref{St-L_cond_class}, and the class of admissible coefficients is much larger
if one uses the quasi-differential machinery. This can be illustrated with an example. 
\begin{example}\label{ex_t^a}
Consider the differential expression  \eqref{St-L expr} with  $p(t) = t^{\alpha}$ and $q(t) = ct^\beta$, and assume $c\neq 0$. 
The conditions \eqref{St-L_cond_class} are reduced to the set of the inequalities $\alpha < 1$ and $\beta > -1$, 
while the conditions \eqref{St-L_cond_GM} hold for
\[ \alpha < 1 \text{ and } \beta > \max\Big\{\alpha - 2, \dfrac{\alpha - 3}{2}\Big\}. \]
So we see that the use of the quasi-derivatives allows one to consider 
the Sturm--Liouville expressions with any power singularity of the potential $q$
if it is compensated by an appropriate function $p$. 
\qed
\end{example}

\begin{remark}
The formulas  \eqref{St-L reg} for the quasi-derivatives contain a certain arbitrariness
due to the non-uniqueness of the function $Q$ which is only determined up to a constant.
However, one can show that if $\widetilde{Q} := Q + c$, 
for some constant $c \in \mathbb{C}$, then
$L_{\operatorname{max}}(Q) =  L_{\operatorname{max}}(\widetilde{Q})$ and
$L_{\operatorname{min}}(Q) =  L_{\operatorname{min}}(\widetilde{Q})$,
i.e. the maximal and minimal operators do not depend 
on the choice of $c$. \qed
\end{remark}

One can easily see that the expression 
\[ l^+(y) = -(\overline{p}y')' + \overline{q}y \]
defines the quasi-differential expression which is formally adjoint to one 
generated by \eqref{St-L expr}. 
It brings up the associated maximal and minimal operators $L^+_{\operatorname{max}}$ and $L^+_{\operatorname{min}}$.
Theorem \ref{thm_L_adjoint} shows that if $p$ and $q$ in \eqref{St-L expr} are real-valued,
then the operator 
$L_{\operatorname{min}} = L^+_{\operatorname{min}}$ is symmetric. 

It is well-known that for the particular case $p \equiv 1$ and
$q \in  L_2([a,b], \mathbb{C})$ one has
\[
\operatorname{Dom}(L_{\operatorname{max}}) = W^2_2([a,b], \mathbb{C}) \subset C^1([a,b], \mathbb{C}).
\]
The following example shows that in some cases all functions
in $\operatorname{Dom}(L_{\operatorname{max}})\setminus\{0\}$ 
are non-smooth. 

\begin{example}\label{ex_Non-smooth Dom}
Consider the differential expression \eqref{St-L expr} with 
\[
p(t) \equiv 1, \qquad q(t) = \sum\limits_{\mu \in \mathbb{Q}\cap(a,b)}{\alpha_\mu \delta(t - \mu)},
\]
where $\mathbb{Q}$ is the set of real rational numbers and 
\[
\alpha_\mu \neq 0 \text{ for all } \mu\in\mathbb{Q}\cap(a,b), \text{ and } \sum\limits_{\mu\in\mathbb{Q}\cap(a,b)}|\alpha_\mu| < \infty.
\]
Then one can take
\[
Q(t) = \sum\limits_{\mu\in\mathbb{Q}\cap(a,b)}{\alpha_\mu H(t - \mu)},
\]
with $H(t)$ being Heaviside function, and $Q$ is a function of a bounded variation 
having discontinuities at every rational point of $(a,b)$. Therefore,
for every subinterval $[\alpha, \beta] \subset (a,b)$ and any
$y \in \operatorname{Dom}(L_{\operatorname{max}})\,\cap\, C^1([\alpha,\beta], \mathbb{C})$
we have
\[
y'(\mu_+) - y'({\mu_-}) = \alpha_\mu y(\mu),\,\, \mu \in\mathbb{Q}\cap[\alpha,\beta]. 
\]
Then $\alpha_\mu y(\mu) = 0$ for all $\mu \in\mathbb{Q}\cap[\alpha,\beta]$, which gives
$y(\mu) = 0$, and the density of  $\{\mu\}\cap[\alpha,\beta]$ in $[\alpha,\beta]$ implies
$y(t) = 0$ for all $t \in [\alpha,\beta]$.
\qed
\end{example}

Now consider the expression
\[
l(y) = i^m y^{(m)}(t) + q(t)y(t),  \qquad m \geq 3,
\]
assuming that 
\begin{equation}\label{m_cond_GM}
\begin{split}
 &q = Q^{(k)}, \quad 1 \leq k \leq \left[\frac{m}{2}\right], \\
 &Q \in 
  \begin{cases}
   L_2([a,b], \mathbb{C}), \, m = 2n, \, k = n; \\[6pt]
   L_1([a,b], \mathbb{C}) \text{ otherwise, }
  \end{cases}
 \end{split}
\end{equation}
where the derivatives of $Q$ are understood in the sense of distributions. 
Introduce the quasi-derivatives as follows:
\begin{equation}\label{m reg}
\begin{split}
&D^{[r]}y = y^{(r)}, \qquad 0 \leq r \leq m - k - 1;\\
&D^{[m - k + s]}y = (D^{[m - k + s - 1]}y)' + i^{-m}(-1)^{s} {k\choose s} \,Q D^{[s]}y, \qquad 0 \leq s \leq k - 1;\\
&D^{[m]}y = \begin{cases} 
            (D^{[m - 1]}y)' + i^{-m}(-1)^{k} { k \choose k} Q D^{[k]}y, & 1 \leq k < m/2, \\
            (D^{[m - 1]}y)' + Q D^{[\frac{m}{2}]}y + (-1)^{\frac{m}{2} + 1} Q^2 y, & m = 2n = 2k;
           \end{cases}
\end{split}
\end{equation}
where $k \choose j$ are the binomial coefficients. 
It is easy to verify that for sufficiently smooth functions $Q$ the equality $l(y) = i^m D^{[m]} y$ holds.
Also one can easily see that, under assumptions \eqref{m_cond_GM},
all the coefficients of the quasi-derivatives \eqref{m reg} are integrable functions. 
The Shin--Zettl matrix corresponding to \eqref{m reg} has the form 
\begin{equation}\label{m matr}
A(t):=\left ( \begin{array}{cccccccc}
0 &1 &0 &\ldots &0 &\ldots &0 &0 \\
0 &0 &1 &\ldots &0 &\ldots &0 &0 \\
\vdots &\vdots &\vdots &\vdots &\vdots &\vdots &\vdots &\vdots\\
- i^{-m} {k\choose 0} Q &0 &0 &\ldots &0 &\ldots &0 &0 \\
0& i^{-m} {k \choose 1} Q &0 &\ldots &0 &\ldots &0 &0 \\
\vdots &\vdots &\vdots &\vdots &\vdots &\vdots &\vdots &\vdots\\
0 &0 &0 &\ldots &0 &\ldots &0 &1 \\
(-1)^{\tfrac{m}{2}} Q^2 \delta_{2k, m} &0 &0 &\ldots &i^{-m}(-1)^{k + 1} { k\choose k} Q &\ldots &0 &0 \\
\end{array}\right),
\end{equation}
where $\delta_{ij}$ is the Kronecker symbol. 
Similarly to the previous case the initial formal differential expression \eqref{m expr} can be defined 
in the quasi-differential form
\[
l[y] := i^{m} D^{[m]} y,
\]
and it generates the corresponding quasi-differential operators 
$L_{\operatorname{min}}$ and $L_{\operatorname{max}}$.

\begin{remark}
Again, the formulas for the quasi-derivatives depend on the choice of the antiderivative $Q$ of order $k$ of
the distribution $q$ which is not only defined up to a
a polynomial of order $\leq k - 1$. 
However, one can show that the maximal and minimal operators do not depend 
on the choice of this polynomial.\qed
\end{remark}

For $k = 1$ the above regularization was proposed in \cite{SavShkal-99},
and for even $m$ they were announced in \cite{MirzShkal-11}.
The general case is presented here for the first time.
Note that if the distribution $q$ is real-valued, then the operator $L_{\operatorname{min}}$ is symmetric.

\section{Extensions of symmetric quasi-differential operators}

Throughout the rest of the paper we assume the Shin--Zettl matrix is formally self-adjoint, i.e.
$A = A^{+}$. The associated quasi-differential expression $l(y)$ is then
formally self-adjoint, $l(y) = l^+(y)$, and the minimal quasi-differential operator  $L_{\operatorname{min}}$
is symmetric with equal deficiency indices by Theorem \ref{thm_L_adjoint}. 
So one may pose a problem of describing (by means of boundary triplets) various classes of 
extensions of  $L_{\operatorname{min}}$ in  $L_2([a,b], \mathbb{C})$.

For the reader's convenience we give a very short summary of the theory of boundary triplets
based on the results of Rofe-Beketov \cite{RofeB-69-eng} and Kochubei \cite{Koch-75}, 
see also the monograph \cite{Gorb-book-eng} and references therein. 


Let $T$ be a closed densely defined symmetric operator in a Hilbert space $\mathcal{H}$ 
with equal (finite or infinite) deficiency indices.

\begin{definition}[\cite{Gorb-book-eng}]\label{PGZdef}
The triplet $\left( H, \Gamma _1 ,\Gamma _2 \right)$, where $H$ is an auxiliary Hilbert space and
$\Gamma_1$, $\Gamma_2$ are the linear maps from $\operatorname{Dom}(T^*)$ to $H$,
is called  a \emph{boundary triplet} for $T$, if the following two conditions are satisfied:
\begin{enumerate}
\item for any $ f,g \in \operatorname{Dom} \left( {L^*} \right)$ there holds
\[
\left( {T^ *  f,g} \right)_\mathcal{H} - \left( {f,T^ *  g}
\right)_\mathcal{H} = \left( {\Gamma_1 f,\Gamma_2 g} \right)_H  -
\left( {\Gamma_2 f,\Gamma_1 g} \right)_H,
\]
\item for any $ g_1, g_2 \in
H$ there is a vector $ f\in \operatorname{Dom} \left( {T^*} \right)$ such that 
$ \Gamma_1 f = g_1$ and $ \Gamma_2 f = g_2$.
\end{enumerate}
\end{definition}
The above definition implies that $ f \in \operatorname{Dom} \left( {T} \right)$ 
if and only if $\Gamma_1f = \Gamma_2f = 0$.
A boundary triplet $\left( {H,\Gamma _1 ,\Gamma _2 } \right)$ with $\operatorname{dim} H = n$ 
exists for any symmetric operator $T$ with equal non-zero deficiency indices  $(n, n)$\, $(n \leq \infty)$,
but it is not unique.

Boundary triplets may be used to describe all maximal dissipative, maximal accumulative and self-adjoint 
extensions of the symmetric operator in the following way.
Recall that a densely defined linear operator $T$ on a complex Hilbert space $\mathcal{H}$ is called
\emph{dissipative} (resp. \emph{accumulative}) if 
\[
\Im \left( Tf, f \right)_\mathcal{H} \geq 0 \quad (\text{resp.} \leq 0), \quad \text{for all } f \in \text{Dom} (T)
\]
and it is called \emph{maximal dissipative} (resp. \emph{maximal accumulative}) if, in addition,
$T$ has no non-trivial dissipative/accumulative extensions in $\mathcal{H}.$
Every symmetric operator is both dissipative and accumulative, and every self-adjoint operator
is a maximal dissipative and maximal accumulative one.
Thus, if one has a symmetric operator $T$, then one can state the problem of describing 
its maximal dissipative and maximal accumulative extensions.
According to Phillips' Theorem \cite{Phil} (see also \cite[p. 154]{Gorb-book-eng})
every maximal dissipative or accumulative extension of a symmetric operator is a restriction of its 
adjoint operator.  Let $\left(H, \Gamma _1,\Gamma _2 \right)$ be a boundary triplet for $T$.
The following theorem is proved in \cite{Gorb-book-eng}.

\begin{theorem}\label{thm_Gorb}
If $K$ is a contraction on $H$, then the restriction of $T^*$ to the set of the vectors
$f \in \operatorname{Dom} \left(T^*\right)$ satisfying the condition  
\begin{equation}\label{absrt_ext_diss}
\left( {K - I} \right)\Gamma_{1} f + i\left( {K + I} \right)\Gamma_{2} f = 0
\end{equation} 
or 
 \begin{equation}\label{absrt_ext_akk}
\left( {K - I} \right)\Gamma_{1} f - i\left( {K + I} \right)\Gamma_{2} f = 0
\end{equation}
is a maximal dissipative, respectively, maximal accumulative extension of $T$. 
Conversely, any maximal dissipative (maximal accumulative) extension of $L$ is the restriction 
of $T^*$ to the set of vectors $f \in \operatorname{Dom} \left(T^*\right)$, 
satisfying \eqref{absrt_ext_diss} or \eqref{absrt_ext_akk}, respectively,
and the contraction $K$ is uniquely 
defined by the extension. 
The maximal symmetric extensions of  $T$ are described by the conditions  
\eqref{absrt_ext_diss} and \eqref{absrt_ext_akk}, where $K$ is an isometric operator. 
These conditions define a self-adjoint extension if $K$ is unitary. 
\end{theorem}

\begin{remark}\label{remark_K}
Let $K_1$ and $K_2$ be the unitary operators on $H$ and let the boundary conditions 
 $$ \left( {K_1 - I} \right)\Gamma_{1}y + i\left( {K_1 + I} \right)\Gamma_{2} y = 0$$
and
 $$ \left( {K_2 - I} \right)\Gamma_{1}y - i\left( {K_2 + I} \right)\Gamma_{2} y = 0$$
define self-adjoint extensions. 
These are two different bijective parameterizations, which reflects the fact
that each self-adjoint operators is maximal dissipative and a maximal accumulative one at the same time.
The extensions, given by these boundary conditions coincide if $K_1 = K_2^{-1}$.
Indeed, the boundary conditions can be written in another form: 
\begin{gather*}
K_1\left(\Gamma _{1}y + i \Gamma _{2}y \right) = \Gamma _{1}y - i\Gamma _{2}y, \qquad 
	  \Gamma _{1}y - i\Gamma _{2}y \in \operatorname{Dom}(K) = H,\\
K_2\left(\Gamma _{1}y - i \Gamma _{2}y \right) = \Gamma _{1}y + i\Gamma _{2}y, \qquad 
	  \Gamma _{1}y + i\Gamma _{2}y \in \operatorname{Dom}(K) = H,
\end{gather*}
and the equivalence of the boundary conditions reads as $K_1K_2 =  K_2K_1 = I$.\qed
\end{remark}

Let us get back to the quasi-differential operators.
The following result is crucial for the rest of the paper as it allows to apply the boundary triplet machinery
to the symmetric minimal quasi-differential operator $L_{\operatorname{min}}$.

\begin{lemma}
  \label{PGZth} 
  Define linear maps $\Gamma_{[1]}$, $\Gamma_{[2]}$  from $\operatorname{Dom}(L_{\operatorname{max}})$ to
$\mathbb{C}^{m}$ as follows:
for $m = 2n$ and $n \geq 2$ we set
\begin{equation}\label{Gamma 2n}
\Gamma_{[1]}y := i^{2n}
 \left(
   \begin{array}{c}
   - D^{[2n - 1]}y(a),\\
   ...,\\
   (-1)^nD^{[n]}y(a),\\
   D^{[2n - 1]}y(b),\\
   ...,\\
   (-1)^{n - 1}D^{[n]}y(b)
   \end{array}
 \right),
 \quad\Gamma_{[2]}y :=
  \left(
   \begin{array}{c}
   D^{[0]}y(a),\\
   ...,\\
   D^{[n - 1]}y(a),\\
   D^{[0]}y(b),\\
   ...,\\
   D^{[n - 1]}y(b)
   \end{array}
  \right)
\end{equation}
and for $m = 2n + 1$ and $n \in \mathbb{N}$ we set
\begin{equation}\label{Gamma 2n+1}
\Gamma_{[1]}y := i^{2n + 1}
   \left(
    \begin{array}{c}
    - D^{[2n]}y(a),\\
    ...,\\
    (-1)^{n} D^{[n + 1]}y(a),\\
    D^{[2n]}y(b),\\
    ....,\\
    (-1)^{n - 1}D^{[n + 1]}y(b),\\
    \alpha D^{[n]}y(b) + \beta D^{[n]}y(a)
    \end{array}
   \right),
  \quad
 \Gamma_{[2]}y :=
   \left(
   \begin{array}{c}
   D^{[0]}y(a),\\
   ...,\\
   D^{[n -1]}y(a),\\
   D^{[0]}y(b),\\
   ...,\\
   D^{[n - 1]}y(b),\\
   \gamma D^{[n]}y(b) + \delta D^{[n]}y(a)
   \end{array}
  \right),
\end{equation}
where
\[
\alpha = 1, \quad
\beta = 1, \quad
\gamma = \frac{(-1)^n}{2} + i,\quad
 \delta = \frac{(-1)^{n + 1}}{2} + i.
 \]
Then $(\mathbb{C}^{m}, \Gamma_{[1]}, \Gamma_{[2]})$ is a boundary triplet for $L_{\operatorname{min}}$.
\end{lemma}

\begin{remark}
The values of the coefficients $\alpha$, $\beta$, $\gamma$, $\delta$
for the odd case may be replaced by an arbitrary set of numbers satisfying the conditions
\begin{equation}\label{PGZ coef}
\begin{gathered}
    \alpha\overline{\gamma} + \overline{\alpha}\gamma = (-1)^{n},\quad
    \beta\overline{\delta} + \overline{\beta}\delta = (-1)^{n + 1}, \quad
    \alpha\overline{\delta} + \overline{\beta}\gamma = 0, \\
    \beta\overline{\gamma}  + \overline{\alpha}\delta = 0, \quad
    \alpha\delta - \beta\gamma \neq 0.
\end{gathered}
\end{equation}
\end{remark}

\begin{proof}
We need to check that the triplet $(\mathbb{C}^{m}, \Gamma_{[1]}, \Gamma_{[2]})$ 
satisfies the conditions $1)$ and $2)$
in Definition \ref{PGZdef} for  $T=L_{\operatorname{min}}$ 
and $\mathcal{H} = L_2\left([a,b], \mathbb{C}\right)$.
Due to Theorem \ref{thm_L_adjoint}, $L^*_{\operatorname{min}} = L_{\operatorname{max}}$.

Let us start with the case of even order. 
Due to Lemma \ref{lm_Lagrange}, for $m =2n$:
\[\left( {L_{\operatorname{max}}y,z} \right) - \left( y,L_{\operatorname{max}}z\right)  =
i^{2n}\sum\limits_{k = 1}^{2n}(-1)^{k - 1}{D^{[2n - k]}y\cdot\overline {D^{[k -1]}z}}\left|_{t = a}^{t = b}
\right..\]

Denote
\[
\Gamma_{[1]} =: \left(\Gamma_{1a}, \Gamma_{1b}\right),\quad
\Gamma_{[2]} =: \left(\Gamma_{2a}, \Gamma _{2b}\right),
\]
where
\begin{align*}
\Gamma_{1a}y &= i^{2n}\left( - D^{[2n - 1]}y(a), ... , (-1)^{n} D^{[n]}y(a)\right),\\
\Gamma_{1b}y &= i^{2n}\left(D^{[2n - 1]}y(b), ... , (-1)^{n - 1} D^{[n]}y(b)\right), \\
\Gamma_{2a}y &= \left(D^{[0]}y(a), ..., D^{[n - 1]}y(a)\right),\\
\Gamma_{2b}y &= \left(D^{[0]}y(b), ..., D^{[n - 1]}y(b)\right).
\end{align*}
One calculates 
\begin{align*}
&\left( {\Gamma _{1a} y,\Gamma _{2a} z} \right)  = i^{2n}
    \sum\limits_{k = 1}^{n}(-1)^{k}{D^{[2n - k]}y(a)\cdot\overline {D^{[k -1]}z(a)}},\\
&\left({\Gamma _{2a} y, \Gamma _{1a} z} \right) = i^{2n}
    \sum\limits_{k = n + 1}^{2n}(-1)^{k - 1}{D^{[2n - k]}y(a)\cdot\overline {D^{[k -1]}z(a)}},\\
&\left( {\Gamma _{1b} y,\Gamma _{2b} z} \right)  = i^{2n}
    \sum\limits_{k = 1}^{n}(-1)^{k - 1}{D^{[2n - k]}y(b)\cdot\overline {D^{[k -1]}z(b)}},\\
&\left({\Gamma _{2b} y, \Gamma _{1b} z} \right) = i^{2n}
    \sum\limits_{k = n + 1}^{2n}(-1)^{k}{D^{[2n - k]}y(b)\cdot\overline {D^{[k -1]}z(b)}},
\end{align*}
which results in
\begin{align*}
\left( {\Gamma _{[1]} y,\Gamma _{[2]} z} \right) &= i^{2n}\sum\limits_{k = 1}^{n}
(-1)^{k - 1}{D^{[2n - k]}y\cdot\overline {D^{[k -1]}z}}\left|_{t = a}^{t = b}\right.,\\
\left({\Gamma _{[2]} y,\Gamma _{[1]} z} \right) &= i^{2n}\sum\limits_{k = n + 1}^{2n}
(-1)^{k}{D^{[2n - k]}y\cdot\overline {D^{[k -1]}z}}\left|_{t = a}^{t = b}\right.,
\end{align*}
and this means that the condition $1)$ of the Definition \ref{PGZdef} is fulfilled,
and the surjectivity condition $2)$ is true due to Lemma \ref{lm_surj}.

The case of odd order is treated similarly.
Due to Lemma \ref{lm_Lagrange}, for $m = 2n + 1$ we have
\[
\left( {L_{\operatorname{max}}y,z} \right) - \left( y,L_{\operatorname{max}}z\right)  =
i^{2n + 1}\sum\limits_{k = 1}^{2n + 1}(-1)^{k - 1}{D^{[2n - k]}y\cdot\overline {D^{[k -1]}z}}\left|_{t = a}^{t = b}
\right..
\]

Denote
\[
\Gamma_{[1]} =: \left(\Gamma_{1a}, \Gamma_{1b}, \Gamma_{1ab}\right),\quad
\Gamma_{[2]} =: \left(\Gamma_{2a}, \Gamma_{2b}, \Gamma_{2ab}\right),
\]
where 
\begin{align*}
\Gamma_{1a}y &= i^{2n + 1}\left( - D^{[2n]}y(a), ... , (-1)^{n} D^{[n + 1]}y(a)\right),\\
\Gamma_{1b}y &= i^{2n + 1}\left(D^{[2n]}y(b), ... , (-1)^{n + 1} D^{[n + 1]}y(b)\right), \\
\Gamma _{1ab}y &= i^{2n + 1}\left(\alpha D^{[n]}y(b) + \beta D^{[n]}y(a)\right),\\
\Gamma _{2a}y &= \left(D^{[0]}y(a), ..., D^{[n - 1]}y(a)\right),\\
\Gamma _{2b}y &= \left(D^{[0]}y(b), ..., D^{[n - 1]}y(b)\right),\\
\Gamma _{2ab}y &= \gamma D^{[n]}y(b) + \delta D^{[n]}y(a).
\end{align*}
One calculates
\begin{align*}
&\left( {\Gamma _{1a} y,\Gamma _{2a} z} \right)  = i^{2n + 1}
    \sum\limits_{k = 1}^{n}(-1)^{k - 1}{D^{[2n - k]}y(a)\cdot\overline {D^{[k -1]}z(a)}},\\
&\left({\Gamma _{2a} y, \Gamma _{1a} z} \right) = i^{2n + 1}
    \sum\limits_{k = n + 2}^{2n + 1}(-1)^{k}{D^{[2n - k]}y(a)\cdot\overline {D^{[k -1]}z(a)}},\\
&\left( {\Gamma _{1b} y,\Gamma _{2b} z} \right)  = i^{2n + 1}
    \sum\limits_{k = 1}^{n}(-1)^{k - 1}{D^{[2n - k]}y(b)\cdot\overline {D^{[k -1]}z(b)}},\\
&\left({\Gamma _{2b} y, \Gamma _{1b} z} \right) = i^{2n + 1}
    \sum\limits_{k = n + 2}^{2n + 1}(-1)^{k}{D^{[2n - k]}y(b)\cdot\overline {D^{[k -1]}z(b)}},\\
&\left( {\Gamma_{1ab} y,\Gamma_{2ab} z} \right) - \left( {\Gamma _{2ab} y,\Gamma _{1ab} z} \right) =\\
&\qquad = i^{2n + 1}(-1)^{n} \left(D^{[n]}y(b)\cdot\overline {D^{[n]}z(b)} - D^{[n]}y(a)\cdot\overline {D^{[n]}z(a)}\right),
\end{align*}
which shows that the condition $1)$ of Definition \ref{PGZdef} is satisfied. 
Now take arbitrary vectors $f_1~=~(f_{1,k})_{k = 0}^{2n},\,f_2~=~(f_{2,k})_{k = 0}^{2n} \in \mathbb{C}^{2n + 1}$.
The last condition in \eqref{PGZ coef} means that the system
\[
\left\{
\begin{array}{l}
\alpha \beta_n + \beta \alpha_n = f_{1, n}\\
\gamma \beta_n + \delta \alpha_n = f_{2, n}\\
\end{array}
\right.
\]
has a unique solution $(\alpha_n,\beta_n)$.
Denoting
\begin{align*}
\alpha_k &:= f_{1, k}, &\beta_k &:= f_{2, k} && \text{for }k < n,\\
\alpha_k & := (-1)^{2n + 1 - k}f_{1, k},& \beta_k &:= (-1)^{2n - k}f_{2, k}&& \text{for } n + 1 < k < 2n
\end{align*}
we obtain two vectors 
$(\alpha _0 ,\alpha _1, ..., \alpha _{m - 1}),(\beta _0 ,\beta _1, ..., \beta_{m - 1})\in\mathbb{C}^m
\equiv \mathbb{C}^{2n+1}$.
By Lemma \ref{lm_surj}, 
there exists a function ${y \in \text{Dom}(L_{\operatorname{max}})}$ such that 
\[
D^{[k]}y(a) = \alpha _k , \quad D^{[k]}y(b) = \beta _k, \quad k = 0,1, ..., m - 1,
\]
and due to above special choice of $\alpha$ and $\beta$ one has
$\Gamma_{[1]}y=f_1$ and $\Gamma_{[2]}y=f_2$, so the surjectivity condition 
of Definition \ref{PGZdef} holds.
\end{proof}

For the sake of convenience, we introduce the following notation. Denote by $L_K$ the restriction of 
$L_{\operatorname{max}}$ onto the set of the functions
$y(t) \in \operatorname{Dom}(L_{\operatorname{max}})$ satisfying the homogeneous boundary condition 
in the canonical form
\begin{equation} \label{diss_ext}
\left( K - I \right)\Gamma_{[1]} y + i\left( K + I \right)\Gamma_{[2]} y = 0.
\end{equation}
Similarly, denote by $L^K$ the restriction of $L_{\operatorname{max}}$ onto the set of the functions
$y(t) \in \operatorname{Dom}(L_{\operatorname{max}})$ satisfying the boundary condition
\begin{equation} \label{akk_ext}
 \left( K - I \right)\Gamma_{[1]} y - i\left( K + I \right)\Gamma_{[2]} y = 0.
\end{equation}
Here  $K$ is an arbitrary bounded operator on the Hilbert space $\mathbb{C}^{m}$, 
and the maps $\Gamma_{[1]}$ è $\Gamma_{[2]}$ are defined by the formulas \eqref{Gamma 2n} or \eqref{Gamma 2n+1}
depending on $m$.
Theorem \ref{thm_Gorb} and Lemma \ref{PGZth} lead to the following description  
of extensions of  $L_{\operatorname{min}}$.

\begin{theorem}\label{thm_ext}
Every $L_K$ with $K$ being a contracting operator in $\mathbb{C}^{m}$,
 is a maximal dissipative extension of   $L_{\operatorname{min}}$.
Similarly every $L^K$ with $K$ being a contracting operator in $\mathbb{C}^{m}$,
is a maximal accumulative extension of the operator $L_{\operatorname{min}}$.
Conversely, for any maximal dissipative (respectively, maximal accumulative) extension $\widetilde{L}$ 
of the operator $L_{\operatorname{min}}$ there exists a contracting operator $K$ such that 
$\widetilde{L} = L_K$\,\, (respectively, $\widetilde{L} = L^K$). 
The extensions $L_K$ and $L^K$ are self-adjoint if and only if $K$ is a unitary operator on $\mathbb{C}^{m}$. 
These correspondences between operators $\{K\}$ and the extensions $\{\widetilde{L}\}$ are all bijective.
\end{theorem}

\begin{remark}
The self-adjoint extensions of a symmetric minimal quasi-differential operator were described by means 
of the Glasman-Krein-Naimark theory in the work \cite{Zettl-75} and several subsequent papers. 
However, the description by means of boundary triplets has important advantages, namely,
it gives a bijective parametrization of extensions by unitary operators,
and one can describe the maximal dissipative and the maximal accumulative extensions
in a similar way.
\end{remark}

\section{Real extensions}

Recall that a linear operator $L$ acting in $L_2([a,b], \mathbb{C})$ is called \textit{real} if:
\begin{enumerate}
\item  For every function $f$ from  $\operatorname{Dom}(L)$ the complex conjugate function $\overline{f}$ 
also lies in $\operatorname{Dom}(L)$.
\item The operator $L$ maps complex conjugate functions into complex conjugate functions, that is 
$L(\overline{f}) = \overline{L(f)}$.
\end{enumerate}

If the minimal quasi-differential operator is real, one arrives at the natural question
on how to describe its real extensions. The following theorem holds.

\begin{theorem}
Let $m$ be even, and let the entries of the Shin--Zettl matrix $A = A^+$ be real-valued, then
the maximal and minimal quasi-differential operators $L_{\operatorname{max}}$ and $L_{\operatorname{min}}$ 
generated by $A$ are real. All real maximal dissipative and maximal accumulative extensions of the real symmetric quasi-differential operator 
$L_{\operatorname{min}}$ of the even order are self-adjoint.
The self-adjoint extensions $L_K$ or $L^K$ are real if and only if the unitary matrix $K$ is symmetric.
\end{theorem}

\begin{proof}
As the coefficients of the quasi-derivatives are real-valued functions, one has
\[
D^{[i]}\overline{y} = \overline{D^{[i]}y}, \qquad i = \overline{1, 2n},
\]
which implies $l(\overline{y}) = \overline{l(y)}$.
Thus for any $y \in \operatorname{Dom}(L_{\operatorname{max}})$ we have
\[
D^{[i]}\overline{y} \in AC([a,b], \mathbb{C}), \,\,
i = \overline{1, 2n - 1}, \quad l(\overline{y})  \in L_2([a,b], \mathbb{C}), \quad
L_{\operatorname{max}}(\overline{y}) = \overline{L_{\operatorname{max}}(y)}.
\]
This shows that the operator $L_{\operatorname{max}}$ is real.
Similarly, for $y \in \operatorname{Dom}(L_{\operatorname{min}})$ we have
\[
D^{[i]}\overline{y}(a) =  \overline{D^{[i]}y}(a) = 0, \quad D^{[i]}\overline{y}(b) =  \overline{D^{[i]}y}(b) = 0, 
\quad i = \overline{1, 2n - 1},
\]
which proves that $L_{\operatorname{min}}$ is a real as well.

Due to the real-valuedness of the coefficients of the quasi-derivatives, the equalities \eqref{Gamma 2n} imply
\[
\Gamma_{[1]}\overline{y} = \overline{\vphantom{\Gamma^1}\Gamma_{[1]}y}, \quad  
\Gamma_{[2]}\overline{y} = \overline{\vphantom{\Gamma^1}\Gamma_{[2]}y}.
\]
As the maximal operator is real, any of its restrictions satisfies the condition $2)$ 
of the above definition of a real operator, so we are reduced to check the condition $1)$.

Let $L_K$ be an arbitrary real maximal dissipative extension given by the boundary conditions \eqref{diss_ext},
then for any $y \in \operatorname{Dom}(L_K)$ the complex conjugate $\overline{y} $ satisfies \eqref{diss_ext} too, 
that is 
\[
\left( K - I \right)\Gamma_{[1]} \overline{y}  + i\left( K + I \right)\Gamma_{[2]} \overline{y}  = 0.
\] 
By taking the complex conjugates we obtain
\[
\left(\overline{K\vphantom{K^1}} - I\right)\Gamma_{[1]} y  - 
	i\left(\overline{K\vphantom{K^1}} + I\right)\Gamma_{[2]}y  = 0,
\]
and $L_K \subset L^{\overline{K}}$ due to Theorem \ref{thm_ext}.
Thus, the dissipative extension $L_K$ is also accumulative,
which means that it is symmetric. But $L_K$ is a maximal
dissipative extension of $L_{\operatorname{min}}$. 
As the deficiency indices of $L_{\operatorname{min}}$ are finite, 
the operator $L_K = L^{\overline{K}}$ must be self-adjoint.
Furthermore, due to Remark \ref{remark_K} the equality $L_K = L^{\overline{K}}$ is equivalent
to $K^{-1} = \overline{K}$. As $K$ is unitary, we have $K^{-1} = \overline{K^T}$, which gives
$K= K^T$. 
In a similar way one can show that a maximal accumulative extension $L^K$ is real if and only if it is 
self-adjoint and $K$ = $K^T$. 
\end{proof}

\section{Separated boundary conditions}

Now we would like to discuss the extensions defined by the so-called separated boundary conditions.

Denote by $\mathbf{f_a}$ the germ of a continuous function $f$ at the point $a$.
We recall that the boundary conditions that define an operator $L \subset L_{\operatorname{max}}$ are called \emph{separated} if 
for any $y \in \operatorname{Dom}(L)$ and any
$g, h \in \operatorname{Dom}(L_{\operatorname{max}})$ with
\[
\mathbf{g_a} = \mathbf{y_a},\quad  \mathbf{g_b} = 0,\quad
\mathbf{h_a} = 0,\quad  \mathbf{h_b} = \mathbf{y_b}.
\]
we have $g, h \in \operatorname{Dom}(L)$.

The following statement gives a description of the operators $L_K$ and $L^K$ with
separated boundary conditions in the case of an even order $m = 2n$ .

\begin{theorem}\label{thm_divided adj}
The boundary conditions \eqref{diss_ext} and \eqref{akk_ext} defining $L_K$ and $L^K$ respectively 
are separated if and only if
the matrix $K$ has the block form 
\begin{equation} \label{separable cond}
K = \left(
\begin{array}{cc}
  K_a & 0 \\
  0 & K_b \\
\end{array}
\right),
 \end{equation}
where $K_a$ and $K_b$ are $n\times n$ matrices.
\end{theorem}

\begin{proof}
We consider the operators $L_K$ only, the case of $L^K$ can be considered in a similar way. 

We start with the following observation.
Let $y, g, h \in \operatorname{Dom}(L_{\operatorname{max}})$. 
Then one can prove by induction that $\mathbf{y_a} = \mathbf{g_a}$ if and only if 
$\mathbf{D^{[k]}y_a} = \mathbf{D^{[k]}g_a}$, $k = 0, 1, ..., m$ 
and, similarly, $\mathbf{y_b} = \mathbf{h_b}$ if and only if 
$\mathbf{D^{[k]}y_b} = \mathbf{D^{[k]}h_b}$, $k = 0, 1, ..., m$.
Therefore, the equality $\mathbf{y_a} = \mathbf{g_a}$ implies $\Gamma_{1a}y = \Gamma_{1a}g$ and $\Gamma_{2a}y = \Gamma_{2a}g$,
and the equality $\mathbf{y_b} = \mathbf{h_b}$ implies $\Gamma_{1b}y = \Gamma_{1b}h$ and $\Gamma_{2b}y = \Gamma_{2b}h$. 

If $K$ has the form \eqref{separable cond}, then the boundary
condition \eqref{diss_ext} can be rewritten as a system:
\[
\left\lbrace
\begin{aligned}
(K_a - I)\Gamma_{1a}y + i(K_a + I)\Gamma_{2a}y& = 0,
\\
 -(K_b - I)\Gamma_{1b}y + i(K_b + I)\Gamma_{2b}y & = 0,
\end{aligned}
 \right.
 \]
and these boundary conditions are obviously separated.

Inversely, let the boundary conditions \eqref{diss_ext} be separated. 
Let us represent $K \in \mathbb{C}^{2n \times 2n}$ in the block form
$$K = \left(%
\begin{array}{cc}
  K_{11} & K_{12} \\
  K_{21} & K_{22} \\
\end{array}%
\right).$$
with $n\times n$ blocks $K_{jk}$. We need to show that $K_{12} = K_{21} = 0$.
The boundary conditions \eqref{diss_ext} take the form 
\[ \left\lbrace
\begin{aligned}
(K_{11} - I)\Gamma_{1a}y + K_{12}\Gamma_{1b}y + i( K_{11} + I)\Gamma_{2a}y + iK_{12}\Gamma_{2b}y & = 0,\\
K_{21}\Gamma_{1a}y + (K_{22} - I)\Gamma_{1b}y + iK_{21}\Gamma_{2a}y + i(K_{22} + I)\Gamma_{2b}y&  = 0.
\end{aligned}
 \right. \]
By definition, any function $g$ with 
$\mathbf{g_a} = \mathbf{y_a}$ and $\mathbf{g_b} = 0$ must also satisfy this system,
which gives
\[ \left\lbrace
\begin{aligned}
&K_{11} \left[\Gamma_{1a}y + i\Gamma_{2a}y\right]= \Gamma_{1a}y - i\Gamma_{2a}y,\\
&K_{21}\left[\Gamma_{1a}y + i\Gamma_{2a}y\right] = 0.
\end{aligned}
 \right. \]
 Therefore, $\Gamma_{1a}y + i\Gamma_{2a}y \in \operatorname{Ker}(K_{21})$ for any 
$y \in \operatorname{Dom}(L_K)$.

Now rewrite  \eqref{diss_ext} in a parametric form. For any $F = (F_1, F_2) \in \mathbb{C}^{2n}$
 consider the vectors $-i \left(K + I\right)F$ and 
$\left(K - I\right)F$.
Due to Lemma \ref{lm_surj} there is a function $y_F \in \operatorname{Dom}(L_{\operatorname{max}})$ such that 
\begin{equation}\label{ext_parametric}
\left\lbrace
\begin{aligned}
-i \left(K + I\right)F &  = \Gamma _{[1]} y_F,\\
\left(K - I\right)F & = \Gamma _{[2]} y_F.
\end{aligned}\right.
\end{equation}
A simple calculation shows that $y_F$ satisfies the boundary conditions \eqref{diss_ext} and, therefore, 
${y_F \in \operatorname{Dom}(L_K)}$.  
We can rewrite \eqref{ext_parametric} as a system,
\[
\left\lbrace
\begin{aligned}
- i(K_{11} + I)F_1 - iK_{12}F_2 & = \Gamma_{1a}y_F,\\
-iK_{21}F_1 - i(K_{22} + I)F_2 & = \Gamma_{1b}y_F,\\
(K_{11} - I)F_1 + K_{12}F_2 & = \Gamma_{2a}y_F,\\
K_{21}F_1 + (K_{22} - I)F_2 & = \Gamma_{2b}y_F.
\end{aligned}
\right.
\]
The first and the third equations show that  
$\Gamma_{1a}y + i\Gamma_{2a}y = -2iF_1$ for any $F_1 \in \mathbb{C}^{n}$.
Therefore, $\operatorname{Ker}(K_{21}) = \mathbb{C}^{n}$ which means $K_{21} = 0$.
The equality  $K_{12} = 0$ is proved in the same way.
\end{proof}

\section{Generalized resolvents}

Let us recall that a \emph{generalized resolvent} of a closed symmetric operator $L$ 
in a Hilbert space $\mathcal{H}$ is an operator-valued function $\lambda\mapsto R_\lambda$ defined on $\mathbb{C} \setminus \mathbb{R}$
which can be represented as
\[
R_\lambda
f = P^+ \left( L^+ - \lambda I^+\right)^{- 1}f, \quad f \in \mathcal{H},
\]
where $L^+$ is a self-adjoint extension  $L$ which acts a certain Hilbert space $\mathcal{H}^+\supset\mathcal{H}$,
$I^+$ is the identity operator on $\mathcal{H}^+$, 
and $P^+$ is the orthogonal projection operator from $\mathcal{H}^+$ onto $\mathcal{H}$.
It is known \cite{Ahiezer-eng} that an operator-valued function 
$R_\lambda$ 
is a generalized resolvent of a symmetric operator $L$ if and only if
it can be represented as
\[
\left( R_\lambda f, g \right)_\mathcal{H} = \int_{-\infty}^{+\infty}\frac{d\left(F_\mu f, g\right)}{\mu - \lambda},
  \quad f, g \in \mathcal{H},
\] 
where $F_\mu$ is a generalized spectral function of the operator $L$, i.e.
$\mu\mapsto F_\mu$ is an operator-valued function $F_\mu$ defined on $\mathbb{R}$ and taking
values in the space of continuous linear operators in $\mathcal{H}$
with the following properties:
\begin{enumerate}
\item  For $\mu_2 > \mu_1$, the difference $F_{\mu_2} - F_{\mu_1}$
 is a bounded non-negative operator.
\item $F_{\mu +} = F_\mu$ for any real $\mu$.
\item For any $x \in \mathcal{H}$ there holds
\[ \lim\limits_{\mu \rightarrow - \infty}^{}||F_\mu x ||_\mathcal{H} = 0,
 \quad \lim\limits_{\mu \rightarrow + \infty}^{} ||{F_\mu x - x} ||_\mathcal{H} = 0.\]
\end{enumerate}

The following theorem provides a description of all generalized resolvents of
 the operator $L_{\operatorname{min}}$.

\begin{theorem}\label{thm_gener resolvent}
$1)$ Every generalized resolvent $R_\lambda$ of the operator $L_{\operatorname{min}}$
in the half-plane $\operatorname{Im}\lambda < 0$
acts by the rule $R_\lambda h = y$, where $y$ is the solution of the boundary-value problem 
\begin{gather*}
l(y) = \lambda y + h,\\
 \left( {K(\lambda) - I} \right)\Gamma _{[1]} f + i\left( {K(\lambda) + I} \right)\Gamma _{[2]} f = 0.
\end{gather*}
Here $h(x) \in L_2([a,b], \mathbb{C})$ and
$K(\lambda)$ is an $m\times m$ matrix-valued function which is holomorph
in the lower half-plane and satisfy $||K(\lambda)|| \leq 1$.

$2)$ In the half-plane $\operatorname{Im}\lambda > 0$, every generalized resolvent of 
$L_{\operatorname{min}}$ acts by $R_\lambda h = y$, 
where $y$ is the solution of the boundary-value problem 
\begin{gather*}
l(y) = \lambda y + h,\\
 \left( {K(\lambda) - I} \right)\Gamma _{[1]} f - i\left( {K(\lambda) + I} \right)\Gamma _{[2]} f = 0.
\end{gather*}
Here $h(x) \in L_2([a,b], \mathbb{C})$ and 
$K(\lambda)$ and
$K(\lambda)$ is an $m\times m$ matrix-valued function which is holomorph
in the lower half-plane and satisfy $||K(\lambda)|| \leq 1$.

The parametrization of the generalized resolvents by the matrix-valued functions $K$ is bijective.
\end{theorem}

\begin{proof}
The Theorem is just an application of Lemma \ref{PGZth} and Theorem 1 and Remark 1 in \cite{Bruk-76}
which prove a description of generalized resolvents in terms of boundary triplets. 
Namely, one requires to take as an auxiliary Hilbert space $\mathbb{C}^m$ and as the operator 
$\gamma y := \{\Gamma_{[1]}y, \Gamma_{[2]}y\}$.

\end{proof}

\end{document}